\documentclass[11pt,regno]{amsart}
\usepackage{euscript,graphicx,psfrag,epsfig}
\newtheorem{theorem}{Theorem}
\newtheorem{lemma}{Lemma}
\newtheorem{corollary}{Corollary}
\newtheorem{proposition}{Proposition}
\theoremstyle{definition}
\newtheorem{example}{Example}
\newtheorem{remark}{Remark}
\newtheorem*{remark*}{Remark}

\newcommand{\eqdef}{\stackrel{\scriptscriptstyle\rm def}{=}}

\DeclareMathOperator{\diam}{diam}
\DeclareMathOperator{\dist}{Dist}
\DeclareMathOperator{\interior}{int}
\DeclareMathOperator{\Leb}{Leb}

\DeclareMathOperator{\Lip}{Lip}

\def\bN{\mathbb{N}}

\def\bR{\mathbb{R}}

\def\cM{\EuScript{M}}

\def\cL{\EuScript{L}}

\DeclareMathSymbol{\varnothing}{\mathord}{AMSb}{"3F}
\renewcommand{\emptyset}{\varnothing}

\author{Katrin Gelfert} \address{Institute of Mathematics, Polish Academy of Sciences, ul. \'{S}niadeckich 8, 00-950 Warszawa, Poland}
\email{gelfert@pks.mpg.de}
\urladdr{http://www.pks.mpg.de/~gelfert}
\author{{Micha\l} Rams} \address{Institute of Mathematics, Polish Academy of Sciences, ul. \'{S}niadeckich 8, 00-950 Warszawa, Poland}
\email{m.rams@impan.gov.pl}
\urladdr{http://www.impan.gov.pl/~rams}

\begin{document}

\title[]{Geometry of limit sets for expansive Markov systems}

\begin{abstract}
We describe the geometric and dynamical properties of expansive Markov systems.
\end{abstract}

\begin{thanks}
{This research of K.\,G. was supported by the grant EU FP6 ToK SPADE2. M.~R. was supported by EU FP6 ToK  SPADE2 and by the Polish KBN Grant No 2P0 3A 034 25. We are grateful for discussions with J.~Rodriguez-Hertz   that lead us to Corollary~\ref{cor:uh}.}
\end{thanks}

\maketitle

\section{Introduction}\label{sec:intro}

Consider a smooth map of an interval into itself. If the map is uniformly expanding, then its geometry and dynamical properties are well understood (see~\cite{Pes:96} for a contemporary view and references). The case that there exist attracting periodic points simply leads to the appearance of basins of attraction with a rather simple dynamics.

 The situation which lies in-between those two cases is yet of considerable interest. 
In the field of dynamical systems, a certain class of such interval maps  have been studied by Urba\'nski~\cite{Urb:96}. His results include, in particular, the calculation of the Hausdorff and packing dimensions and Lebesgue, Hausdorff, and packing measures  of the limit set. 
Maps of this type appear also in the study of so-called Markov exceptional minimal sets in foliation theory (see for example~\cite{CanCon:88},~\cite{Mat:},~\cite{BisUrb:07}, see~\cite{Hur:} or~\cite{Wal:04} for extensive surveys of the field).    

The goal of this paper is to describe the geometric properties of the general class of systems studied by Cantwell and Conlon~\cite{CanCon:88}. It includes maps that may have marginally stable periodic orbits and hence are not (uniformly) expanding. As our main working assumption we require the map to be expansive (though not necessarily everywhere expanding). We describe its geometric properties such as the Hausdorff dimension of the limit set in terms of dynamical quantities like entropy, topological pressure, and conformal measures (see Section~\ref{sec:prelim} for the definition and notations). The following is our main result.

\begin{theorem}\label{Mmain}
	Let $\Lambda_Q$ be the limit set of a $C^{1+\Lip}$ expansive and transitive Markov system $(\Lambda_Q,f)$. 		
	Then the Hausdorff dimension of $\Lambda_Q$ satisfies
	\[\begin{split}
	\dim_{\rm H}\Lambda_Q
	=& \inf\{t\ge 0\colon P(-t\log\lvert f'\rvert)=0\}\\
	=&\inf\{t\ge 0\colon \text{there exists a }t\text{-conformal measure supported on }\Lambda_Q\}\\
	=&  \sup\{\dim_{\rm H}\mu\colon\mu \in\cM_{\rm E}, \, h_\mu(f)>0\} .
	\end{split}\]
	 Moreover, $\Lambda_Q$ is of zero Lebesgue measure.
	 
Under the additional assumption that $f$ is (piecewise) real analytic, we have $\dim_{\rm H}\Lambda_Q<1$ and the packing measure in this dimension is positive.
\end{theorem}

Under the assumption of Theorem~\ref{Mmain}, the zero Lebesgue measure property has been proved by Cantwell and Conlon~\cite{CanCon:88}. The analogous dimension formulae of Theorem~\ref{Mmain} have been shown in the $C^{1+\varepsilon}$ case by Urba\'nski~\cite{Urb:96} under additional geometric assumptions. 
We note that we require that $f$ is of class $C^{1+\Lip}$ because our main tool for the proof of Theorem~\ref{Mmain} is the A.~Schwartz lemma. 
The additional analyticity assumption is necessary, as the following example shows.

\begin{example} \label{example}
There exists a $C^\infty$ parabolic iterated function system in $\bR$ with a limit set with Hausdorff dimension equal to 1, but of zero Lebesgue measure. (Full details are given at the end of Section~\ref{sec:5}.) 
\begin{figure}[h]\label{figuras}
      \psfrag{1.0}[b][rr]{$0.2$}
      \epsfig{figure=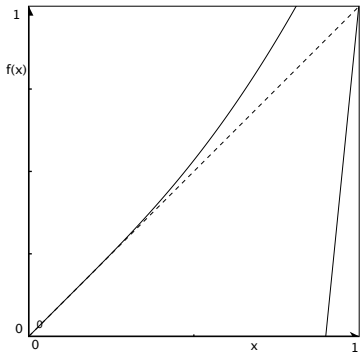, width = 0.4\linewidth}
      \caption{The iterated function system given by~\eqref{ugu}.}
      \label{exfig1}
       \end{figure}
\end{example}

Note that there exists an iterated function system in $\bR$ of smoothness $C^1$ with positive Lebesgue measure (see the example by Bowen~\cite{Bow:75}).

In Section~\ref{sec:hyptime} we introduce the main concept for our approach, so-called hyperbolic instants. We introduce the  set $H$ of points which have infinitely many hyperbolic instants and we prove some of its immediate properties. The conclusions in this section are quite general and, in fact, are valid for a much bigger class of dynamical systems than those considered in this paper. For expansive and transitive Markov systems, most points belong to the set $H$. Namely we have the following theorem.

\begin{theorem} \label{thm:hyp}
	Under assumptions of Theorem \ref{Mmain}, a point belongs to $\Lambda_Q\setminus H$ if and only if it is either a parabolic periodic point or parabolic preperiodic point.
\end{theorem}

We now sketch the contents of the paper.
In Section~\ref{sec:prelim} we introduce the main objects of our study, namely Markov systems, topological pressure, and conformal measures.
In Section~\ref{sec:hyptime} we define the set $H$ and prove some of its main properties. In Section~\ref{sec:4} we prove Theorem~\ref{thm:hyp}. The proof is based on ideas from Cantwell and Conlon~\cite{CanCon:88} and Matsumoto~\cite{Mat:}.
Finally, in Section~\ref{sec:5} we prove Theorem~\ref{Mmain}. For this we use Theorem~\ref{thm:hyp} together with now-standard methods from~\cite{Urb:96,DenUrb:91b}.

\section{Preliminaries}\label{sec:prelim}

We consider a general geometric construction of a Cantor-like set, 
modelled on a Markov geometric construction (a subshift of
finite type).

\subsection{Iterated function systems}
 
Let $I\subset \bR$ be a closed interval and $I_1$, $\ldots$, $I_p$ be a finite
collection of pairwise disjoint closed subintervals of $I$. 
Let $\{ g_i\}_{i=1}^p$ be a collection of bijective $C^{1+\Lip}$ diffeomorphisms
with 
range $R( g_i)\subset \interior(I_i)$ and domain $D( g_i)=\bigcup_{j=1}^p I_j$.
Such a collection is called an \emph{iterated function system}.
We will denote
\[
 g_{i_1\ldots i_n}\eqdef  g_{i_1}\circ\cdots\circ  g_{i_n}.
\]

We consider the associated topological Markov chain $\sigma\colon
\Sigma\to\Sigma$ defined by $\sigma(i_1i_2\ldots)=(i_2i_3\ldots)$ on the
set 
\[
\Sigma
\eqdef\{1,\ldots,p\}^\bN .
\]

We denote $\Sigma_n=\{1,\ldots,p\}^n$ and $\Sigma_*=\bigcup_{n=0}^\infty
\Sigma_n$, where we use the convention $\Sigma_0=\{\emptyset\}$.
For each $(i_1\ldots i_n)\in\Sigma_n$ we define
\[
\Delta_{i_1\ldots i_n}\eqdef  g_{i_1\ldots i_{n-1}}(I_{i_n})
\]
and $\Delta_\emptyset=I$. 

We define
\[
\Lambda 
\eqdef \bigcap_{n=1}^\infty\bigcup_{(i_1\ldots i_n)\in \Sigma_n}\Delta_{i_1\ldots i_n}.
\]
We set
\[
f(x) \eqdef  g_i^{-1}(x) \quad\text{ when }x\in R( g_i).
\]
Given $x\in \Lambda$, we denote by $\Delta_n(x)$ the unique set $\Delta_{i_1
  \ldots i_n}$ containing $x$. 

Given an iterated function system, for each $(i_1i_2\ldots)\in\Sigma$ the sets $\Delta_{i_1\ldots i_n}$ form a descending sequence of non-empty compact sets, and in general the set $\bigcap_{n=1}^\infty \Delta_{i_1\ldots i_n}$ is either a singleton or a non-degenerate interval. We always assume that $\bigcap_{n=1}^\infty \Delta_{i_1\ldots i_n}$ is a singleton, that is, that
\begin{equation}\label{eq1}
d_n\eqdef
\max_{(i_1\ldots i_n)}\lvert\Delta_{i_1\ldots i_n}\rvert\to 0 \text{ as } n\to\infty,
\end{equation}
and in particular that $f|\Lambda$ is an expansive map. Recall that $f|\Lambda$ is called \emph{expansive} if there exists a positive constant $\delta$ such that, for any $x$, $y\in\Lambda$ satisfying $\lvert f^k(x)-f^k(x)\rvert\le \delta$ for every $k\ge 1$, we have $x=y$. We obtain a coding map $\pi\colon\Sigma\to\Lambda$ given by 
\[
\pi(i_1i_2\ldots)
=\bigcap_{n=1}^\infty  g_{i_1}\circ\cdots\circ  g_{i_{n-1}}(I_{i_n})
=\bigcap_{n=1}^\infty \Delta_{i_1\ldots i_n}.
\]
We call the sequence $\pi^{-1}(x)$ the \emph{symbolic expansion} of $x\in\Lambda$.

We remark that for an iterated function system for which the maps $ g_1$, $\ldots$, $ g_p$ are not uniformly contracting on $\Lambda$,  the thermodynamic formalism of the
associated symbolic dynamics does not apply since the coding map $\pi$ does
not in general preserve the H\"older continuity 
of potentials.   

We further remark that it would be enough to assume only that $g_i(\bigcup I_j)\subset I_i$.
For such systems one can increase all $I_i$ and expand domains of $g_i$ in such a way as to obtain $C^{1+\Lip}$ (but not necessarily $C^2$, even if the original system was smoother!) expansive iterated function system $(\hat{I_i}, \hat{g}_i)$ with the same $\Lambda$ and satisfying $\hat{g}_i(\bigcup \hat{I}_j)\subset \interior \hat{I}_i$.

\subsection{Markov systems}

We call an $n$-tuple $(i_1\ldots i_n)\in\Sigma_n$ a \emph{word} of \emph{length} $n$.
Let $Q\subset \Sigma_*$ be a finite set 
and denote by $l(Q)$ the maximal length of words from $Q$.
We denote by $\Sigma_Q$ the set of all infinite sequences from $\Sigma$ that do not contain any of the words from $Q$ as a sub-word:
\[
\Sigma_Q
\eqdef\left\{(i_1i_2\ldots)\in\Sigma\colon (i_ni_{n+1}\ldots i_{n+k-1})\notin Q
                \text{ for every }n,k\in\bN\right\}.
\]

Naturally, the same set $\Sigma_Q$ may be obtained for different choices of $Q$.
We will denote
\[
l(\Sigma_Q)\eqdef\min l(Q),
\]
where the minimum is taken over all possible choices of $Q$ resulting in the same $\Sigma_Q$.
We call the (not necessarily unique) set $Q$ for which the minimum is achieved a \emph{representation} of $\Sigma_Q$.
The proof of the following lemma is left to the reader.

\begin{lemma} \label{lem:incl}
Assume $Q$ is a representation of $\Sigma_Q$. For every $Q'\supset Q$ satisfying $l(Q')=l(Q)$ we have 
\[
l(\Sigma_{Q'})\leq l(\Sigma_Q).
\]
\end{lemma}

From the definition of $\Sigma_Q$ it follows that $\sigma(\Sigma_Q)\subset \Sigma_Q$, but this inclusion needs not to be an equality. We say that $\Sigma_Q$ is \emph{completely invariant} if and only if  $\sigma(\Sigma_Q)=\Sigma_Q$, and we call $\sigma|\Sigma_Q$ a \emph{subshift of finite type} in this case.
Given $(j_1\ldots j_m)\in\Sigma_\ast$, we denote by $j_1\ldots j_m \Sigma_Q$ the set of all infinite sequences from $\Sigma$ given by
\[
j_1\ldots j_m \Sigma_Q \eqdef
\{(j_1\ldots j_mi_1 i_2\ldots)\colon (i_1i_2\ldots )\in\Sigma_Q\}.
\]

\begin{proposition} \label{prop:irred}
For any $\Sigma_Q$ there exists a completely invariant set $\Sigma_{Q'}$ with $l(\Sigma_{Q'})\leq l(\Sigma_Q)$, and there exists a finite set $\widetilde{Q}\subset \Sigma_*$ such that
\[
\Sigma_{Q'}\subset \Sigma_Q \subset \bigcup_{(j_1j_2\ldots j_m)\in \widetilde{Q}} j_1j_2\ldots j_m \Sigma_{Q'}.
\]
\end{proposition}

\begin{proof}
Let $\Sigma_Q$ (with representation $Q$) be not completely invariant, that is, let us assume that
there exists such $(i_1i_2\ldots)\in \Sigma_Q$ that $(i\,i_1i_2\ldots)\notin \Sigma_Q$ for all $i\in \{1,\ldots,p\}$. 
Denoting $n=l(Q)$, we see that all the words $(i\,i_1i_2\ldots i_{n-1})$,
$i\in\{1,\ldots,p\}$, must contain a word from $Q$ as a sub-word. 
At the same time, $(i_1i_2\ldots i_{n-1})$ does not contain any word from $Q$
as a sub-word. 
Let $Q'=Q\cup \{(i_1i_2\ldots i_{n-1})\}$.
Note that we have $l(Q')= l(Q)$ and
\[
\Sigma_{Q'}\subset \Sigma_Q \subset \Sigma_{Q'} \cup i_1\Sigma_{Q'}.
\]
Moreover, by Lemma \ref{lem:incl} 
\[
l(\Sigma_{Q'})\leq l(\Sigma_Q).
\]

If $\Sigma_{Q'}$ is not completely invariant then we can repeat this procedure.
As there exist only finitely many choices of finite sets $Q''\subset\Sigma_\ast$ satisfying $l(Q'')\leq n$, we can repeat this procedure only finitely many times. This proves the proposition.
\end{proof}

Denote
\begin{equation}\label{defQ}
\Lambda_Q\eqdef \pi(\Sigma_Q).
\end{equation}
We have $f(\Lambda_Q)\subset \Lambda_Q$ (with equality if and only if $\Sigma_Q$ is completely invariant).
Given $Q$ and $x\in \Lambda_Q$, we call the sequence $(i_1i_2\ldots)\in \Sigma$ \emph{admissible} for $x$ if
\begin{equation}\label{admin} 
	 g_{i_n\ldots i_1}(x) \in \Lambda_Q \quad\text{ for all }n\in \bN.
\end{equation}
If $\Sigma_Q$ is completely invariant, then for every point $x\in \Lambda_Q$ there exists an admissible sequence.
The proof of the following lemma will be left to the reader.

\begin{lemma} \label{lem:adm}
If a sequence is admissible for $x$, then it is admissible for every $y\in \Delta_n(x)\cap \Lambda_Q$ with  $n=l(\Sigma_Q)-1$.
\end{lemma}

The system $(\Lambda_Q, f,  g_1,\ldots, g_p)$ is called a \emph{Markov system}.
The partition $\Lambda_Q = \bigcup_{k=1}^p J_k$, where $J_k = I_k\cap \Lambda_Q$, has the property that $f(J_k)$ either contains $J_\ell$ or is disjoint with $J_\ell$ for every $k$, $\ell$. 
A partition with this property is called {\it Markov}.
We note that, in the context of dynamical systems theory, one usually considers Markov systems which are topologically conjugate to a subshift of finite type. Wheras here we also want to include the more general case, which in particular will be part of our constructions in Section~\ref{sec:4}.  

From now on we will always assume that $\Sigma_Q$ is completely invariant.

The Markov system needs not, in general, to be transitive.
However, if the system is not transitive then it is a disjoint union of finitely many transitive Markov systems and we can study each of them separately.
In what follows, we consider only transitive Markov systems.

\subsection{Topological pressure}\label{sec:press}

Let $\varphi$ be a continuous function on $\Lambda_Q$. The \emph{topological pressure} of $\varphi$ (with respect to $f|\Lambda_Q$) is defined by
\begin{equation}\label{wurm}
  P(\varphi) \eqdef
  \lim_{n\to\infty}\frac{1}{n} \log \sum_{(i_1\ldots i_n)}
      \exp \max_{x\in\Delta_{i_1\ldots i_n}} S_n\varphi(x),
\end{equation}
where here and in the sequel the sum is taken over the cylinders with non-empty intersection with the set $\Lambda_Q$.
The existence of the limit follows easily from the fact that the sum constitutes a sub-multiplicative sequence. Moreover, the value $P(\varphi)$ does not depend on the particular Markov partition that we use in its definition.  

Denote by $\cM$ the family of $f$-invariant Borel probability measures on $\Lambda_Q$. 
By the variational principle we have
\begin{equation}\label{wirm}
 P(\varphi)= 
      \max_{\mu\in\cM}\left( h_\mu(f)+\int_{\Lambda_Q} \varphi d\mu\right),
\end{equation}
where $h_\mu(f)$ denotes the entropy of $f$ with respect to $\mu$ (see~\cite{Wal:81}).

\begin{lemma}\label{lem1}
	$f$ has tempered distortion on $\Lambda_Q$, that is, there exists a positive sequence $(\rho_n)_n$ decreasing to $0$ such that for every $n$ we have
	\begin{equation}\label{*t35}
  	\sup_{(i_1\ldots i_n)}\sup_{x,y\in\Delta_{i_1\ldots i_n}}
  	\frac{\lvert (f^n)'(x)\rvert}{\lvert(f^n)'(y)\rvert} 
	\le e^{n\rho_n}.
	\end{equation}
\end{lemma}

\begin{proof}
  By Lipschitz continuity of $f|\Lambda_Q$ we have
  \[
   \log\sup_{(i_1\ldots i_n)}\sup_{x,y\in\Delta_{i_1\ldots i_n}}
   \frac{\lvert (f^n)'(x)\rvert}{\lvert(f^n)'(y)\rvert}   \le c\sum_{k=1}^nd_k ,  
  \]
  for some positive number $c$ and for $d_k$ defined in~\eqref{eq1}. By 
  the condition~\eqref{eq1} the sequence $(d_k)_k$ is decreasing to $0$,
  which implies the claimed property.
\end{proof}

Given $t\in\bR$, we define the function $\varphi_t\colon\Lambda_Q\to\bR$ by
\begin{equation}\label{varphit}
        \varphi_t(x) \eqdef
         -t\log\lvert f'(x)\rvert.
\end{equation}
The tempered variation property~\eqref{*t35} ensures in particular that in the definition of $P(\varphi_t)$ in~\eqref{wurm} one can replace the maximum by the minimum or, in fact, by any intermediate value. 
Given $x\in\Lambda_Q$ we denote by $\chi(x)$ the \emph{Lyapunov exponent}
at $x$
\[
\chi(x)\eqdef\limsup_{n\to\infty}\frac{1}{n}\log\lvert (f^n)'(x)\rvert\,.
\] 

\begin{proposition}
  The function $t\mapsto P(\varphi_t)$ is a continuous, convex, and non-increasing function of $\bR$. $P(\varphi_t)$ is negative for large $t$ if and only if there exist no $f$-invariant probability measures with zero Lyapunov exponent.   
\end{proposition}  

\begin{proof}
  The claimed properties follow immediately from general facts about the
  pressure together with the variational principle~\eqref{wirm}. 
\end{proof}

We define
\[
t_0
\eqdef\inf\{t>0\colon P(\varphi_t)\le 0\}.
\]

\begin{proposition}\label{lem:perpress}
  Suppose that $f|\Lambda_Q$ is transitive, then for every $t$
  \begin{equation}\label{hig}
  P(\varphi_t) = 
  \lim_{n\to\infty}\frac{1}{n}\log 
  \sum_{\tiny
    \begin{matrix}f^n(x)=x,\\
    x\in \Lambda_Q
    \end{matrix}} \lvert(f^n)'(x)\rvert^{-t}.
  \end{equation}
\end{proposition}

\begin{proof}
By transitivity, for each $I_j$, $j=1$, $\ldots$, $p$, there exists a
cylinder $\Delta(j)\eqdef\Delta_{i_1\ldots i_{s_j}}\subset I_j$ such that
$f^{s_j}(\Delta_{i_1\ldots i_{s_j}} \cap \Lambda_Q)=\Lambda_Q$.  Put $s\eqdef 
\max\{s_j\colon 1\le j \le p\}$. For every $n\ge 1$ denote 
\[
\Sigma_n^0\eqdef
\{(i_1\ldots i_n)\colon f^n(\Delta_{i_1\ldots i_n}\cap \Lambda_Q)=\Lambda_Q \}.
\]
Recall the choice of the sequence $(\rho_n)_n$ given in~\eqref{*t35}.
Notice that by the tempered distortion property, for any $n\ge 1$ we have
\begin{eqnarray}
\sum_{\tiny
    \begin{matrix}f^n(x)=x,\\
    x\in \Lambda_Q
    \end{matrix}}\lvert (f^n)'(x)\rvert^{-t}
&\ge &\sum_{(i_1\ldots i_n)\in\Sigma_n^0}
     \min_{x\in\Delta_{i_1\ldots i_n}}\lvert (f^n)'(x)\rvert^{-t} \notag\\
&\ge &e^{-nt\rho_n}
     \sum_{(i_1\ldots i_n)\in\Sigma_n^0}
    \lvert (f^n)'(x_{i_1\ldots i_n})\rvert^{-t} ,\label{tiilda}
\end{eqnarray}
where $x_{i_1\ldots i_n}\in\Delta_{i_1\ldots i_n}$ is arbitrarily chosen. 

Given $k>1$, we have
\[\begin{split}
  &\sum_{(i_1\ldots i_k j_1\ldots j_s) \in\Sigma_{k+s}^0}
  \lvert (f^{k+s})'(x_{i_1\ldots i_k j_1\ldots j_s})\rvert^{-t}\\
  &=\sum_{(i_1\ldots i_k)}
    \sum_{\tiny
    \begin{matrix}
      (j_1\ldots j_s) \in \Sigma_s^0,\\
      f^k(\Delta_{i_1\ldots i_k}) \supset \Delta_{j_1\ldots j_s}
    \end{matrix}}
	\lvert (f^{k+s})'(x_{i_1\ldots i_k j_1\ldots j_s})\rvert^{-t} \\
& \ge\sum_{(i_1\ldots i_k)}
	 \sum_{\tiny
    \begin{matrix}
      (j_1\ldots j_s) \in \Sigma_s^0,\\
      f^k(\Delta_{i_1\ldots i_k}) \supset \Delta_{j_1\ldots j_s}
    \end{matrix}}
    \min_{x\in\Delta_{i_1\ldots i_k}}\lvert (f^k)'(x)\rvert^{-t}\,
	\lvert (f^s)'(f^k(x_{i_1\ldots i_kj_1\ldots j_s}))\rvert^{-t}\\
& \ge  C \sum_{(i_1\ldots i_k)} 
\min_{x\in\Delta_{i_1\ldots i_k}}\lvert (f^k)'(x)\rvert^{-t}\\
& \ge C e^{-kt\rho_k}\sum_{(i_1\ldots i_k)} 
\max_{x\in\Delta_{i_1\ldots i_k}}\lvert (f^k)'(x)\rvert^{-t}
\end{split}\]
(remember that each sum is taken over the cylinders which intersect $\Lambda_Q$),
where $C$ is some positive constant depending entirely only on $f^s$.  

Since every $\Delta_{i_1\ldots i_n}$ contains at most one periodic point of period $n$, we have
\begin{equation}\label{iiilda}
\sum_{\tiny\begin{matrix}f^n(x)=x,\\x\in\Lambda_Q\end{matrix}} \lvert (f^n)'(x)\rvert^{-t}
\le \sum_{(i_1\ldots i_n)}
    \max_{x\in\Delta_{i_1\ldots i_n}}\lvert (f^n)'(x)\rvert^{-t}.
\end{equation}
From here the claimed property follows.
\end{proof}

We remark that the above proof used some ideas from~\cite[Section 5]{Yur:98}.

\subsection{Conformal measures}\label{sec:confmeas}

Denote by $\cL_\varphi \psi(x) \eqdef \sum_{f(y)=x, y\in \Lambda_Q}e^\varphi(y)\psi(y)$, $x\in\Lambda_Q$, the \emph{Ruelle-Perron-Frobenius transfer operator} associated to $\varphi$, which is acting on  the space  of continuous functions $\psi\colon \Lambda_Q\to\bR$.
By the Schauder-Tychonov theorem, there exists a fixed point of the map $\nu\mapsto (\cL_\varphi^\ast\nu(1))^{-1}\cL_\varphi^\ast\nu$. This fixed point, say $\nu$, is a conformal measure with respect to the function $\exp(P(\varphi)-\varphi)$, that is, it satisfies
\[
\nu(f(A)) = \int_A e^{P(\varphi)-\varphi(x)} d\nu(x).
\]
for every Borel subset $A$ of $\Lambda_Q$ such that $f|_A$ in
injective (see~\cite{DenUrb:91} for more details on conformal measures) . 


In the particular case that in the definition of the transfer operator we
consider the potential $\varphi_t$ defined in~\eqref{varphit} with $t$
such that $P(\varphi_t)=0$, such a measure $\nu_t$ satisfies
\begin{equation}\label{conf}
\nu_t(f(A)) = \int_A \lvert f'(x)\rvert^t d\nu_t(x).
\end{equation}
A measure satisfying~\eqref{conf} is called \emph{$t$-conformal} for
$f|\Lambda_Q$.  
The following lemma follows immediately from the definition, the proof will be left to the reader.

\begin{lemma} \label{nonzero}
If $f|\Lambda_Q$ is transitive, the $t$-conformal measure is strictly positive on every cylinder.
\end{lemma}

We denote by 
\[
t_c\eqdef\inf\{t\ge0 \colon t\text{-conformal measure exists on } \Lambda_Q\}.
\]

\section{Points with infinitively many hyperbolic instants}\label{sec:hyptime}

We will \emph{a priori} not require that $f|\Lambda_Q$ is expanding neither that $f|\Lambda_Q$ has bounded distortion. However, in the following, we derive that a rather large set of points has already good expanding and distortion properties. 

Let us denote by 
\[
\dist g|_Z \eqdef \sup_{x,y\in Z}\frac{\lvert g'(x)\rvert}{\lvert g'(y)\rvert} 
\]
the maximal distortion of a map $g$ on a set $Z$.
We denote by $H$ the set of points $x$ in $\Lambda_Q$ which have the property
that there exist positive constants $c_1$, $c_2$, a monotonously increasing
sequence of natural numbers $(n_k)_k$ and a positive sequence $(r_k)_k$
decreasing to $0$ (all, in general, depending on $x$) such that for every
$k\ge 1$ the map $f^{n_k}$ is defined on $B(x,r_k)$ and we have 
\[
\diam f^{n_k}(B(x,r_k)) >c_1 \text{ and }
\dist f^{n_k}|_{B(x,r_k)}<c_2.
\]  
We call the numbers $n_k$ \emph{hyperbolic instants} for $x$. We remark the following statements of this section remain true if we consider $\Lambda$ instead of $\Lambda_Q$, however, we keep our restricted point of view.

We want to relate the above concept to the concept of hyperbolic times, which requires slightly stronger expansion properties of a given orbit. A number $n\in\bN$ is called a \emph{hyperbolic time} for a point $x\in \Lambda_Q$ with exponent $\alpha$ if 
\[
\lvert (f^k)'(f^{n-k}(x))\rvert 
\ge e^{k\alpha} \text{ for every }1\le k \le n.
\]
We denote by $H_{\rm HT}(\alpha)$ the set of points $x\in\Lambda_Q$ for which there exist infinitely many hyperbolic times with exponent $\alpha$. Obviously, $H_{\rm HT}(\alpha)\subset  H_{\rm HT}(\alpha')$ if $\alpha'<\alpha$.

The following is an immediate consequence of the generator condition~\eqref{eq1} and Lemma~\ref{lem1}.

\begin{lemma}\label{genull}
	We have $\chi(x)\ge 0$ for every $x\in\Lambda_Q$.	
\end{lemma}

We now show some immediate properties of the set $H$. 

\begin{proposition}\label{prop:dist}
  We have 
  \[
  \left\{x\in\Lambda_Q\colon \chi(x)>0 \right\} = \bigcup_{\alpha>0} H_{\rm HT}(\alpha)
  \subset H.
  \]
  A periodic point $x$ or its preimage is in $H$ if and only if $\chi(x)>0$. 
\end{proposition}

\begin{remark}
    Under the prerequisites of Urba\'nski in~\cite{Urb:96} every point
    in $\Lambda$ except parabolic fixed points and their preimages belong to
    $H$. 
    Thus, $H$ can contain points with zero Lyapunov exponent. 
\end{remark}

\begin{remark}
	It follows from the proof of Proposition~\ref{prop:dist} that the constants $c_1$ and $c_2$ can, in fact, be chosen uniformly on any of the sets $\{x\in H\colon \chi(x)\ge\chi_0>0\}$.
\end{remark}

\begin{proof}[Proof of Proposition~\ref{prop:dist}]
 Since $f$ locally extends to a $C^{1+\Lip}$ diffeomorphism, for a given $\alpha>0$ there exists $c_1>0$ such that  
  \begin{equation}\label{huhu}
    \sup_{y,z\in I_i}\max_{\lvert f(y)-f(z)\rvert\le c_1}
    \frac{\lvert f'(y)\rvert}{\lvert f'(z)\rvert} 
    \le e^{\alpha/2}.
  \end{equation}
We choose $c_1$ so small that two points from $\Lambda_Q$ in distance smaller than $c_1$ must belong to the same $I_i$.

  Let $n$ be a hyperbolic time for $x\in\Lambda_Q$ with exponent $\alpha>0$.
  We now prove by induction that  for every $y$ such that $f^n(y)\in B(f^n(x),c_1)$ 
  for every $0\le k\le n$ we have
  \begin{equation}\label{huhuhu}
  \lvert f^{n-k}(x)-f^{n-k}(y)\rvert e^{k\alpha/2} \le \lvert f^n(x)-f^n(y)\rvert.
  \end{equation}
  For $y$,~\eqref{huhuhu} is true for $k=0$. Assume now that we have this property for each $k=0$, $\ldots$, $\ell$, and hence that 
  \[
  \lvert f^{n-\ell}(x)-f^{n-\ell}(y)\rvert\le \lvert f^n(x)-f^n(y)\rvert e^{-\ell\alpha/2}
  <c_1e^{-\ell\alpha/2}<c_1.
  \]
   With~\eqref{huhu} and the fact that $\lvert (f^{\ell+1})'(f^{n-\ell-1}(x))\rvert\ge e^{(\ell+1)\alpha}$ this implies that $\lvert f^{n-\ell-1}(x)-f^{n-\ell-1}(y)\rvert \le \lvert f^n(x)-f^n(y)\rvert e^{-(\ell+1)\alpha/2}$. This shows~\eqref{huhuhu}. 
  
    In particular,~\eqref{huhuhu} with $k=n$ implies the first claimed property since by setting $r(n)\eqdef c_1e^{-n\alpha/2}$ we have
  \[
  c_1\le \diam f^n B(x,c_1r(n)).
  \]
  
  We now prove the bounded distortion property. By~\eqref{huhuhu} and
  by Lipschitz continuity of $\log\lvert f'\rvert$ with some Lipschitz constant $c>0$ we have
  \[\begin{split}
  \log\dist f^n|_{B(x,r(n))}
  &= \sup_{y,z\in B(x,r(n))} \log
    \frac{\lvert f'(y)\rvert\lvert f'(f(y))\rvert\cdots\lvert
  f'(f^{n-1}(y))\rvert} 
         {\lvert f'(z)\rvert\lvert f'(f(z))\rvert\cdots\lvert
  f'(f^{n-1}(z))\rvert} \\
  &\le \sum_{k=0}^{n-1} ce^{-(n-k)\alpha/2} \lvert I\rvert
   \le \sum_{k=0}^\infty ce^{-k\alpha/2} \lvert I\rvert\eqdef c_2 
   <\infty
  \end{split}\]
  which proves the inclusion $ H_{\rm HT}(\alpha)\subset H$.

To show that $\bigcup_{\alpha>0} H_{\rm HT}(\alpha)$ coincides with the set of points with positive Lyapunov exponent, one can apply an auxiliary result due to Pliss (see, for example,~\cite{Pli:72}). 
However, we prefer to give a direct proof. 
Let  $x\in\Lambda_Q$ be a point with positive Lyapunov exponent.  
Let us assume that $x\notin  H_{\rm HT}(\chi_0)$ for some  $0<\chi_0<\chi(x)$, in other words, that it has at most finitely many hyperbolic times $n_i$ with exponent $\chi_0$. 
Let $N\eqdef\max_in_i$. 
Let $n>N$. 
Since $n$ is not a hyperbolic time, there exists $k_1<n$ such that 
\[
\lvert (f^{n-k_1})'(f^{k_1}(x))\rvert \le e^{(n-k_1)\chi_0}.
\]  
If $k_1>N$, we can continue and find numbers $k_\ell\le N< k_{\ell-1}<\cdots <k_1<n$ such that
\[
\lvert (f^{k_{i-1}-k_i})'(f^{k_i}(x))\rvert \le e^{(k_{i-1}-k_i)\chi_0}, \quad i=2,\ldots,\ell.
\]  
Hence, we have
\[
\lvert (f^{n-k_i})'(f^{k_i}(x))\rvert\le e^{(n-k_i)\chi_0}.
\]
But this implies that there exists some positive constant $C=C(N,\chi_0)$ such that
\[
\lvert (f^n)'(x)\rvert \le Ce^{n\chi_0} \text{ for every }n>N
\]
and hence $\chi(x)\le\chi_0$, which is a contradiction. Thus, we can conclude that $x$ has infinitely many hyperbolic times with any exponent smaller than $\chi(x)$. 
  
It follows obviously from the first part of this proof that a preimage of a periodic point $x$ is in $H$ if $\chi(x)>0$.
We finally prove that if a periodic point $x$ or its preimage is in $H$ then we must have  $\chi(x)>0$.  Let $x=f^k(x)$ be a point satisfying $\lvert (f^k)'(x)\rvert=1$. Let $y= g_{i_n\ldots i_1}(x)$ be in $H$ and let  $c_1$ and $c_2$ be associated constants and $(n_k)_k$ and $(r_k)_k$ the associated sequences. Then
  \[
  \diam  B(y,r_k) \ge \lvert (f^{n_k})'(y) \rvert^{-1} c_2^{-1}c_1 =
  \lvert (f^n)'(y)\rvert^{-1}  \lvert (f^{n_k-n})'(x)\rvert^{-1} c_2^{-1}c_1 
  \]
  is uniformly bounded from below by some positive constant. In particular this means that $\lvert \Delta_k(y)\rvert$ stays bounded away from zero for
all $k\in \bN$ (a contradiction with \eqref{eq1}).
This finishes the proof of Proposition~\ref{prop:dist}.
\end{proof}

\begin{proposition}\label{prop:dimH}
  We have $\dim_{\rm H} H\le t_c$. 
\end{proposition}

\begin{proof} 
Let $t\ge t_c$ such that there exists a $t$-conformal measure $\mu_t$ for $f|\Lambda_Q$.
Let $x\in H$ and consider the sequences $(n_k)_k$ and $(r_k)_k$ and positive constants $c_1$ and $c_2$ associated to $x$. Note that 
  \begin{equation}\label{man}
  2 r_k\cdot \inf_{y\in B(x,r_k)}\lvert (f^{n_k})'(y)\rvert 
  \le \diam f^{n_k}(B(x,r_k)) \le |I|.
  \end{equation}
  By $t$-conformality of the measure $\mu_t$, we have
  \begin{eqnarray*}
   \mu_t(f^{n_k}(B(x,r_k))) 
  &=& \int_{B(x,r_k)}\lvert (f^{n_k})'(y)\rvert^t d\mu_t(y)\\
  &\le& {c_2}^t\inf_{y\in B(x,r_k)}\lvert (f^{n_k})'(y)\rvert^t
        \mu_t(B(x,r_k)),
  \end{eqnarray*}
  where for the second inequality we used the bounded distortion
  property of the map $f^{n_k}$ at the point $x$. With~\eqref{man} we obtain
  \[
  \mu_t(f^{n_k}(B(x,r_k))) 
  \le \left(\frac{c_2}{2r_k} |I|\right)^t\mu_t(B(x,r_k)).
  \]
As $\diam(f^{n_k}(B(x,r_k)))> c_1$, $\mu_t(f^{n_k}(B(x,r_k)))$ is bounded away from zero uniformly in $k$ by some constant $c$ (by Lemma \ref{nonzero} and a compactness argument)  and hence
  \[
  \frac{\log\mu_t(B(x,r_k))}{\log r_k} 
  \le \frac{-t \log(c_2 |I| \,2^{-1}) + c}{\log r_k} +t.
  \]
  This implies that the lower pointwise dimension of the measure $\mu_t$
  at $x$ can be bounded by $\underline{d}_{\mu_t}(x)\le t$. Since $x\in
  H$ was arbitrary, we obtain $\dim_{\rm H}H\le t$ (see for
  example~\cite{Fal:97}). 
\end{proof}

Note that the conformal measure $\mu_t$ is not necessarily concentrated on the set $H$. But we do not need this property in order to obtain the upper bound of $\dim_{\rm H}H$. 

\begin{proposition}\label{prop:zero}
  The set $H$ has zero Lebesgue measure.
\end{proposition}

\begin{proof}
  We will show that no $x\in H$ is a Lebesgue-density point of $H$, that is, we have
  \begin{equation}\label{hum}
  \liminf_{r\to 0}\frac{\Leb(H\cap B(x,r))}{\Leb(B(x,r))}<1 \text{ for every
  }x\in H,
  \end{equation}
  where $\Leb$ denotes the Lebesgue measure on $\bR$. Since a set is
  nowhere dense if and only if it has zero Lebesgue measure, the assertion then
  follows. 

  Let $x\in H$ with associated sequences $(n_k)_k$ and $(r_k)_k$ and numbers
  $c_1>0$ and $c_2\ge 1$. There exists a positive number $c_3$, which depends only  on $c_1$ but neither on $n_k$ nor $r_k$, such that for every number $n_k$ there exists an interval gap $G_k\subset f^{n_k}(B(x,r_k))\setminus\Lambda_Q$ of length at least $c_3$. Denote $\widetilde G_k=f^{-n_k}(G_k)$. Analogously to~\eqref{man}, by 
  bounded distortion of the map $f^{n_k}$ we have 
  \[
 2  r_k \cdot \lvert (f^{n_k})'(x)\rvert \cdot c_2^{-1}
  \le \diam f^{n_k}(B(x,r_k)\cap \Lambda)\le \lvert I\rvert
  \]
  which implies 
  \[
  c_3
  \le \lvert G_k\rvert 
  \le \lvert \widetilde G_k\rvert \sup_{y\in B(x,r_k)}\lvert (f^{n_k})'(y)\rvert
  \le \lvert \widetilde G_k\rvert \lvert (f^{n_k})'(x)\rvert c_2 
  \le \lvert \widetilde G_k\rvert\, \lvert I\rvert \, \frac{ c_2^2}{2r_k} .
  \]
  Thus we obtain
  \[
  \frac{\Leb(H\cap B(x,r_k))}{\Leb(B(x,r_k))}
  \le \frac{r_k (1- c_3 c_2^{-2}\lvert I\rvert^{-1})}{r_k} <1
  \]
  for a sequence $(r_k)_k$. From here~\eqref{hum} follows.
\end{proof}

\section{Points with infinitely many hyperbolic instants - geometric description}\label{sec:4}

Our goal in this section is to prove Theorem~\ref{thm:hyp}, that is, to prove that a point belongs to $\Lambda_Q\setminus H$ if and only if it is either a parabolic periodic point or a  parabolic preperiodic point.

We briefly sketch how we are going to prove this result.
We first formulate the A.~Schwartz lemma in our setting. We then identify the family of intervals to which this lemma is applicable. The main step in the proof is done by a bootstrapping argument: by using of the A.~Schwartz lemma we are able to cut out all points that visit a certain cylinder infinitely many times. By doing so, we can restrict ourself to a new, strictly smaller, Markov system and use the A.~Schwartz lemma again. This procedure terminates after finitely many steps.

\begin{proof}[Proof of Theorem~\ref{thm:hyp}]
Since $x\in H$ if and only if $f(x)\in H$, by Proposition~\ref{prop:irred} we can restrict ourself to the case that $\Sigma_Q$ is irreducible. We assume that $Q$ is a representation of $\Sigma_Q$.

We start with a distortion result which is based on the A.~Schwartz lemma~\cite{Sch:63}.
Denoting by ${\rm Lip}\,g$ the Lipschitz constant of $g$, let 
\[
\theta\eqdef \max_{1\le i\le p} {\rm Lip}\,\log\,\lvert g_i'\rvert
\]
and 
\begin{equation}\label{lamdef}
\lambda \eqdef \exp (4 \theta \lvert I\rvert).
\end{equation}

\begin{lemma} \label{lem:sch}
Let $G$ and $K$ be two intersecting intervals which are contained in some $I_i$, $1\le i\le p$, and which satisfy $\lambda\lvert K\rvert \le \lvert G\rvert$. If $(i_1i_2\ldots )\in \Sigma$ satisfies
\[
\sum_{k=1}^\infty \diam  g_{i_k\ldots i_1}(G)\leq 2\lvert I\rvert,
\]
then for every $n\in\bN$ we have
\[
\diam  g_{i_n\ldots i_1}(K) \leq \diam  g_{i_n\ldots i_1}(G) 
\quad\text{ and }\quad
\dist  g_{i_n\ldots i_1}|_{G\cup K} \leq \lambda .
\]
\end{lemma}

\begin{proof}
We prove the statement by induction over $n$.
The assertion is true for $n=1$ because we have $\lvert K\rvert<\lvert G\rvert$ and 
\[
\log(\lvert  g_i'(x)\rvert\lvert  g_i'(y)\rvert^{-1})\le \theta\lvert x-y\rvert <
 2\theta\lvert I\rvert
\]
for every $x$, $y\in G\cup K$.

Let us assume that the assertion is true for all $n<N$.
We have

\[
\dist  g_{i_N\ldots i_1}|_{G\cup K} \leq 
\exp (\theta \sum_{k=0}^{N-1}\diam  g_{i_k\ldots i_1}(G\cup K)) .
\]
By the inductive hypothesis, for all $k<N$ we have 
\[
\diam  g_{i_k\ldots i_1}(G\cup K)
\leq 2\diam  g_{i_k\ldots i_1}(G).
\]
 This implies that
\[
\dist  g_{i_N\ldots i_1}|_{G\cup K} \leq \exp (4 \theta \lvert I\rvert) 
= \lambda
\]
and that
\[
\frac {\diam  g_{i_N\ldots i_1}(K)} {\diam  g_{i_N\ldots i_1}(G)} 
\leq \frac {\lvert K\rvert} {\lvert G\rvert}
\dist  g_{i_N\ldots i_1}|_{G\cup K} \leq 1 .
\]
This proves the lemma.
\end{proof}

We will call an open interval $G$ contained in one of the intervals $I_i$ a \emph{gap} if $G\cap \Lambda_Q = \emptyset$ and if at least one of endpoints of $G$ belongs to $\Lambda_Q$.

We will apply Lemma \ref{lem:sch} to the following geometrical situation:

\begin{lemma} \label{lem:geom}
Let $G\subset I_i$ be a gap, $x\in \Lambda_Q$ be an endpoint of $G$ and let $(i_1i_2\ldots)\in \Sigma$ be an admissible sequence for $x$.
Assume that the points $ g_{i_k\ldots i_1}(x)$ are all distinct for $k\in \bN$.
Then
\[
\sum_{k=1}^\infty \diam  g_{i_k\ldots i_1}(G) \leq 2\lvert I\rvert .
\]
\end{lemma}

\begin{proof}
We will prove that $ g_{i_k\ldots i_1}(G)$ are pairwise disjoint for all such
$k$ for which $ g_{i_k\ldots i_1}$ are both orientation preserving or both orientation reversing (on $I_i$).

Assume that it is not true, i.e., that $ g_{i_k\ldots i_1}(G)$ intersects $ g_{i_l\ldots i_1}(G)$ for $ g_{i_k\ldots i_1}$ and $ g_{i_\ell\ldots i_1}$ both orientation preserving or both orientation reversing.
The points $ g_{i_k\ldots i_1}(x)$ and $ g_{i_\ell\ldots i_1}(x)$ are both right endpoints (or both left endpoints) of $ g_{i_k\ldots i_1}(G)$ and $ g_{i_\ell\ldots i_1}(G)$, correspondingly.
Assume that $k<\ell$.
There are three possible cases.

Case 1: $y= g_{i_k\ldots i_1}(x)\in  g_{i_\ell\ldots i_1}(G)$. 
In this situation $f^\ell(y)\in G$. 
However, $f^\ell(y)=f^{\ell-k}(x)\in \Lambda_Q$, which is a contradiction.

Case 2: $y= g_{i_\ell\ldots i_1}(x)\in  g_{i_k\ldots i_1}(G)$.
In this situation $f^k(y)\in G$.
However, $f^k(y)= g_{i_{\ell-k}\ldots i_1}(x)\in \Lambda_Q$ (by definition of admissible sequence), which is again a contradiction.

Case 3: $ g_{i_k\ldots i_1}(x)= g_{i_\ell\ldots i_1}(x)$ is excluded by the
assumption of the lemma. 
We are done.
\end{proof}

Recall the definition of $d_n$ in~\eqref{eq1}. Let $G$ be a gap and let $n$ be so big that 
\begin{equation} \label{eqn:assump}
d_n \leq \frac{\lvert G\rvert}{3\lambda}.
\end{equation}
We also assume that 
\begin{equation} \label{eqn:assump2}
n\geq l(Q).
\end{equation}
Let $x$ be an endpoint of $G$ which is contained in $\Lambda_Q$. 
Denote by $K$ the $d_n$-neighbour\-hood of $\Delta_n(x)=\Delta_{j_1\ldots j_n}$.
Let $Q'\eqdef Q\cup \{(j_1\ldots j_n)\}$.

\begin{remark}
Before continuing the proof, let us remind you that the range of $g_i$ lies strictly inside $I_i$.
Hence, $\Lambda$ is strictly contained in the interior of $\bigcup g_i(\bigcup I_j)$.
In particular, if $g_{i_k\ldots i_1}|_K$ has distortion bounded by $\lambda$ then for any $y\in g_{i_k\ldots i_1}(\Delta_n(x))$ the map $f^k$ is well defined on $B(y, D|\Delta_{n+k}(y)|/\lambda)$ for some $D$ not depending on $k$.
\end{remark}

\begin{lemma} \label{lem:key}
If the assertion of Theorem \ref{thm:hyp} is true for $\Lambda_{Q'}$, it is
also true for $\Lambda_Q$. 
\end{lemma}

\begin{proof}
We denote by $K_\ast$ the set of points $y\in\Lambda_Q$ for which  $f^k(y)\in
\Delta_n(x)$ for infinitely many $k\in\bN$.
We are going to show first that every point in $K_\ast$ which is not a preimage of a  parabolic periodic point is contained in $H$.

Note first that by Lemma \ref{lem:adm} any sequence admissible for $x$ is
necessarily admissible for all the points in $\Delta_n(x)$. 
There are two cases.

Case 1: $x$ is not a periodic point of $f$. Then for any sequence $(i_1i_2\ldots)$ which is admissible for $x$, the points $ g_{i_k\ldots i_1}(x)$, $k\ge 1$, are pairwise distinct. 
Hence, by Lemma~\ref{lem:geom} we can apply Lemma~\ref{lem:sch} to the sets
$G$ and $K$ and to the maps $ g_{i_k\ldots i_1}$. 
In particular we have 
\begin{equation} \label{eqn:bdp}
\dist  g_{i_k\ldots i_1}|_{G\cup K}\le \lambda
\end{equation}
for every $k\in\bN$.

For any $y\in \Lambda_Q$, if $y\in  g_{i_k\ldots i_1}(\Delta_n(x))$ for some $k$-tuple $(i_k\ldots i_1)$ then $f^{k}(\Delta_{n+k}(y))=\Delta_n(x)$.
It follows from~\eqref{eqn:bdp} that we have
\[
\diam f^k\left(B\left(y,\frac{D \lvert\Delta_{n+k}(y)\rvert}
                                              {\lambda}
                                              \right)\right)
\geq \frac D {\lambda^2} \lvert \Delta_{j_1\ldots j_n}\rvert
\]
and
\[
\dist f^k|_{B\left(y,\frac{D \lvert\Delta_{n+k}(y)\rvert}{\lambda}\right)} \leq \lambda .
\]
Hence, if $f^k(y)\in\Delta_n(x)$ for infinitely many $k$, then $y\in H$. We conclude  $K_\ast\subset H$ in Case 1.

Case 2: $x$ is a periodic point of $f$ of period $m$. We have $x=\pi((\ell_1\ldots \ell_m)^\bN)$ for some $m$-tuple $(\ell_1\ldots \ell_m)$ and 
\begin{equation}\label{equ}
\ell_k=j_k \text{ for }k\leq \min\{m,n\}.
\end{equation}
Let $y\in K_\ast$.

Assume for a moment that there exists $M_0>0$ such that $f^k(y)\in \Delta_{j_1\ldots j_n}$ implies that 
\begin{equation} \label{eqn:a3}
f^{k-m}(y)\in\Delta_{\ell_1\ldots \ell_m}
\end{equation}
for every $k>M_0$.
It means that in the symbolic expansion of $y$ from some moment on any
appearance of $(j_1\ldots j_n)$ is immediately preceded by $(\ell_1\ldots \ell_m)$, that is
\[
y=\pi(\underset{k-m \text{ symbols }}{\ast\,\,\ldots\,\,\ast}\ell_1\ldots \ell_m\, j_1\ldots j_n\ldots)
.\]
From~\eqref{equ} it then follows that the word $(\ell_1\ldots \ell_m j_1 \ldots j_n)$ 
must be preceded by $(\ell_1\ldots \ell_m)$ again, and so on. As $y\in K_\ast$, the word $(j_1\ldots j_n)$ appears infinitely many times, hence the symbolic expansion of $y$ is of the form $(i_1\ldots i_k
(\ell_1\ldots \ell_m)^\bN)$, that is, $y$ is a preimage of the periodic point $x$. 

If $y$ is a preimage of $x$, then, by Proposition~\ref{prop:dist}, $y\in H$ if and only if $x$ is an expanding periodic point. 

If $y$ is not a preimage of $x$ then~\eqref{eqn:a3} is not satisfied, that is, there exist infinitely many numbers $k$ such that $f^k(y)\in \Delta_n(x)$ and 
$f^{k-m}(y)\notin \Delta_{\ell_1\ldots \ell_m}$. 
For any such  $k$, the sequence $(i_k\ldots i_1)$ which satisfies $y= g_{i_k\ldots i_1}(f^k(y))$ is admissible for all points in $\Delta_n(x)$ and does not begin with $(\ell_m\ldots \ell_1)$. 
Hence, the points $ g_{i_\ell\ldots i_1}(x)$ are pairwise distinct for \emph{every} $1 \leq \ell\leq k$.
By Lemma~\ref{lem:geom}, we can apply Lemma~\ref{lem:sch} to the sets $G$ and $K$ and to the maps $ g_{i_k\ldots i_1}$.
We obtain
\[
\diam f^k\left(B\left(y,\frac{D \lvert\Delta_{n+k}(y)\rvert}{\lambda}\right)\right)
\geq \frac D {\lambda^2} \lvert \Delta_{j_1\ldots j_n}\rvert
\] 
and 
\[
\dist f^k |_{B\left(y,\frac{D \lvert \Delta_{n+k}(y)\rvert}{\lambda}\right)} \leq \lambda
\]
like in Case 1.
This implies that all the points in $K_\ast$ which are not preimages of parabolic periodic points belong to $H$ if we consider the Case 2.

Consider now a point $y\in\Lambda_Q$ for which $f^{k_i}(y)\in \Delta_n(x)$ for only finitely many $k_i$. Then $f^k(y)\in \Lambda_{Q'}$ for every $k>\max_ik_i$.
Hence,
\[
\Lambda_Q\subset K_\ast \cup \bigcup_{(i_1\ldots i_m)}  g_{i_1\ldots i_m}(\Lambda_{Q'}),
\]
where the union is taken over some infinite but countable subset of $\Sigma_\ast$. 
In such a situation, to obtain the assertion of Theorem~\ref{thm:hyp} for $\Lambda_Q$ it is enough to prove this assertion for $\Lambda_{Q'}$.
\end{proof}

Lemma \ref{lem:key} gives us the inductive step for the proof of Theorem \ref{thm:hyp}.
We start by choosing $p$ gaps (one in each $I_i$) and $n\in\bN$ such that \eqref{eqn:assump} is satisfied for each gap and that \eqref{eqn:assump2} is satisfied as well.
To a chosen gap $G\subset I_i$ we apply Lemma \ref{lem:key}.
We see that to prove the assertion of Theorem \ref{thm:hyp} for $\Lambda_Q$ it is enough to prove it for some subsystem $\Lambda_{Q'}$.
Using Proposition \ref{prop:irred} we might well assume that $\Sigma_{Q'}$ is irreducible.

Either $\Lambda_{Q'}$ is disjoint from $I_i$ (in this case we proceed to another gap) or we can find a new gap $G'\supset G$.
Its diameter is obviously not smaller than the diameter of $G$ and hence \eqref{eqn:assump} is satisfied for $G'$.
Lemma \ref{lem:incl} implies that \eqref{eqn:assump2} is satisfied for $\Lambda_{Q'}$.
Hence, we can repeat the inductive step.

Since every Markov system which we consider in this chain is strictly contained in the previous one and since $l(\Sigma_Q)$ is bounded from above, the procedure must eventually terminate.
In other words, at some moment we will obtain $\Lambda_{Q'}=\emptyset$.
This will end the proof of Theorem~\ref{thm:hyp}.
\end{proof}

\begin{corollary}\label{cor:finite}
There exist at most finitely many parabolic periodic points.
\end{corollary}

\begin{proof}
In the proof of Theorem \ref{thm:hyp} we used an inductive procedure that
produces at most one parabolic periodic point at every step.
All other periodic points (not the endpoints of gaps on some step of our procedure) belong to $H$ and thus cannot be parabolic.
As the inductive procedure terminates after finitely many iterations, the
number of parabolic periodic points can be at most finite.
\end{proof}

The next corollary follows in the case that $f$ is $C^2$ from the Ma\~{n}e hyperbolicity theorem as well, see~\cite{dMevSt} for the formulation. 

\begin{corollary}\label{cor:no}
If there exist no parabolic periodic points, the system is uniformly hyperbolic.
\end{corollary}

\begin{proof}
If there exist no parabolic periodic points, then, by Theorem~\ref{thm:hyp}, we have $H=\Lambda_Q$.
Every point $x\in H$ has a cylinder neighbourhood $U(x)$ and a time $k(x)$ such that $f^{k(x)}|_{U(x)}$ is uniformly expanding.
As $\Lambda_Q=H$ is compact, from sets $U(x)$ one can choose a finite covering of $\Lambda_Q$.
Hence, $\Lambda_Q$ is equivalent (under a change of symbolic description) to a finite uniformly expanding Markov system.
\end{proof}

We say that the Markov system $(\Lambda_Q,f)$ is $C^1$  \emph{robustly expansive}  if and only if any map $g$  which is a sufficiently small $C^1$ smooth perturbation of $f$ also satisfies~\eqref{eq1}. 

\begin{corollary}\label{cor:uh}
	The Markov system $(\Lambda_Q,f)$ is $C^1$  robustly expansive  if and only if $f|\Lambda_Q$ is uniformly hyperbolic.
\end{corollary}

\begin{proof}
 Let $(\Lambda_Q,f)$ be a $C^{1+\Lip}$ expansive and transitive Markov system.
 By Corollary~\ref{cor:finite}  there can exist at most a finite number of parabolic periodic points for $f|\Lambda_Q$.
 Let  us assume that there exists at least one parabolic periodic point $x\in\Lambda_Q$, and we may assume that $x$ is a fixed point. Then we can find a perturbation $g$ of $f$, which keeps $x$ fixed, which is arbitrarily $C^1$ close to $f$
  and for which there is a neighborhood $U$ of $x$ such that $g(U)\subset {\rm int}(U)$, violating~\eqref{eq1}. 
  Hence $f|\Lambda_Q$ must be uniformly hyperbolic by Corollary~\ref{cor:no}.\\
 If $f|\Lambda_Q$ is uniformly hyperbolic, then any  small $C^1$ perturbation is also hyperbolic, and hence expansive.
\end{proof}

\section{Proof of Theorem~\ref{Mmain}}\label{sec:5}

Given $\mu\in\cM$, we denote by
\[
\chi(\mu)\eqdef \int\log\lvert f'\rvert d\mu 
\]
the \emph{Lyapunov exponent} of $\mu$.

By Theorem~\ref{thm:hyp}, the set $\Lambda_Q\setminus H$ is at most countable and hence has zero Lebesgue measure. Thus, it follows from Proposition~\ref{prop:zero} that the Lebesgue measure of $\Lambda_Q$ is zero. 
  
We now prove the dimension results. From Theorem~\ref{thm:hyp} we know that $\dim_{\rm H}\Lambda_Q\setminus H=0$, and hence
\begin{equation}\label{eq0}
  \dim_{\rm H}\Lambda_Q\le t_c
\end{equation}
follows from Proposition~\ref{prop:dimH}. 
    
    Note that $f|\Lambda_Q$ is an open and expansive map. By~\cite[Theorem 3.12]{DenUrb:91}, for every $t\ge 0$ satisfying $P(-t\log\lvert f'\rvert)=0$ there exists a $t$-conformal measure. This  implies that 
    \begin{equation}\label{eq111}
    t_c\le t_0.
    \end{equation}
    
Denote 
\[
D\eqdef \sup_\mu \dim_{\rm H}\mu,
\] 
where the supremum is taken over all ergodic $f$-invariant measures supported on $\Lambda_Q$ with positive entropy. 
$D$ is commonly refered to as the \emph{dynamical dimension} of $f|\Lambda_Q$.
Assume that $D>0$ and consider $0<t<D$. It follows from~\cite{HofRai:92} that there exists an ergodic $f$-invariant measure $\mu$ with positive entropy such that $t< h_\mu(f)/\chi(\mu)$. This implies
	\[
	0< h_\mu(f)-t\chi(\mu)\le P(-t\log\lvert f'\rvert)
	\]
	and hence $t<t_0$. Note that $0\le t_0$ trivially holds true. This implies $D \le t_0$.

	Notice that for every ergodic $f$-invariant measure $\nu$ we have $\chi(\nu)\ge 0$, by Lemma~\ref{genull}.
	Suppose now that $P(-t\log\lvert f'\rvert)>0$ for some $t\ge 0$. Then, by the variational principle, there exists an ergodic $f$-invariant measure $\nu$ such that $ h_\nu(f)-t\chi(\nu)>0$ and thus $h_\nu(f)>0$. It follows then from~\cite{HofRai:92} that 
	\[
	t<\frac{h_\nu(f)}{\chi(\nu)} =\dim_{\rm H}\nu\le D.
	\]
    	This implies that $P(-t\log\lvert f'\rvert)\le0$ for every $t\ge D$, and hence we have shown that 
	\begin{equation}\label{eq2}
	 t_0= D.
	\end{equation}
Notice that, by definition, we have 
\begin{equation}\label{eq3}
D\le \dim_{\rm H}\Lambda_Q.
\end{equation}
With~\eqref{eq0}--\eqref{eq3} this proves the first part of the assertion.
	
The following construction of conformal measures omitting exceptional points is nowadays standard (see~\cite{Urb:96} or~\cite{DenUrb:91}). We repeat it here because it will be used for Example~\ref{example}. 
	
\begin{proposition}\label{mmart}
	If $f$ is (piecewise) real analytic, then there exists a $t_0$-dimensional conformal measure $\nu$ such that $\nu(\Lambda_Q\setminus H)=0$.
\end{proposition}		
	
\begin{proof}
	Given $n\ge 1$, define
	\[
	Q_n\eqdef Q \cup \bigcup (i_1\ldots i_n),
	\]
	where the union is taken over all words $(i_1\ldots i_n)$ for which $\Delta_{i_1\ldots i_n}=\Delta_n(x)$ for some parabolic periodic point $x\in \Lambda_Q$.
	Recall the definition~\eqref{defQ}. By Proposition~\ref{prop:irred} there exists a set $\Lambda_n\subset \Lambda_{Q_n}$ which is non-empty for large enough $n$ and which is $f$-invariant and, by construction and by Corollary~\ref{cor:no}, uniformly expanding with respect to $f$. By~\cite{Wal:78} there exists a $t_n$-conformal measure supported on $\Lambda_n$, denote
	\[
	t_n\eqdef \dim_{\rm H}\Lambda_n.
	\]
Notice that $(t_n)_n$ is non-decreasing sequence.
Moreover, $t_n$ is the unique zero of the equation $P_{f|\Lambda_n}(\varphi_t)=0$, where $P_{f|\Lambda_n}$ denotes the topological pressure with respect to $f|\Lambda_n$. 	
It follows that $t_n\le t_0$ for every $n\ge1$.  There exists a subsequence $(\nu_{n_k})_k$ and a number $t$ such that $\nu_{n_k}$ converges to some probability measure $\nu$ in the weak$\ast$ topology and that $t_{n_k}$ converges to $t$. Let $A\subset \Lambda_Q$ be a Borel set such that $f|A$ is injective. Note that, by weak$\ast$ convergence and by conformality, we have
\[
\nu(f(A)) = \lim_{k\to\infty}\nu_{n_k}(A) 
= \lim_{k\to\infty} \int_A\lvert f'(x)\rvert^{t_{n_k}} d\nu_{n_k} 
=\int_A\lvert f'(x)\rvert^t d\nu,
\]
and hence $\nu$ is $t$-conformal.  Since there is no $t'$-conformal measure for $t'<t_0$, we must have $t=t_0$.

	We now prove that $\nu(\Lambda_Q\setminus H)=0$. It is sufficient to prove that there are no atoms at parabolic points. We consider a  parabolic fixed point $x\in\Lambda_Q$ at which $f$ is locally orientation preserving, and show that we have $\nu(x)=0$.  The general case then follows from the consideration of some iterate of the map $f$. Let us assume that there exists such a point $x\in\Lambda_Q$. Given $n\ge 1$, denote by $D^\pm_n(x)$ the right and the left interval of the set $\Delta_n(x)\setminus \Delta_{n+1}(x)$, respectively. Note that 
\begin{equation}\label{iter}
f(D^\pm_n(x))=D^\pm_{n-1}(x). 
\end{equation} 
In particular, we have
\[
\sum_{k=0}^\infty \lvert D_k^\pm(x)\rvert 
<\lvert I\rvert.
\]
Given some number $M\ge 1$, from Lemma~\ref{lem:sch} we can conclude that for every $k\ge 1$
\[
\dist f^k|_{D^\pm_{M+k}(x)} \le \lambda ,
\]
where $\lambda$ is given in~\eqref{lamdef}.
Using this property together with the conformality of each of the measures $\nu_n$, for every $k\ge 1$ we derive that
\begin{multline}\label{martin}
\nu_n(D^\pm_M(x)))
= \nu_n(f^k(D^\pm_{M+k}(x)))
= \int_{D^\pm_{M+k}(x)} \lvert (f^k)'\rvert^{t_n}d\nu_n\\
\ge \frac{1}{\lambda}
      \frac{\lvert D^\pm_{M}(x)\rvert^{t_n}}{\lvert D^\pm_{M+k}(x)\rvert^{t_n}}
      \nu_n(D^\pm_{M+k}(x))
\end{multline}
if $M+k\le n$ and $\nu_n(D^\pm_{M+k}(x))) =0$ else, and analogously that
\begin{equation}\label{annne}
\nu_n(D^\pm_M(x))) \le
\lambda
      \frac{\lvert D^\pm_{M}(x)\rvert^{t_n}}{\lvert D^\pm_{M+k}(x)\rvert^{t_n}}
      \nu_n(D^\pm_{M+k}(x)).
\end{equation}
The same estimations are valid for $(\nu,t_0)$ in place of $(\nu_n, t_n)$.
For fixed $j\ge 1$, with~\eqref{iter} we can conclude that
\[
\nu(\Delta_j(x))=\lim_{n\to\infty}\nu_n(\Delta_j(x))
= \lim_{n\to\infty} \left(\sum_{\ell=j}^\infty\nu_n(D_\ell^-(x)) +  
	\sum_{\ell=j}^\infty\nu_n(D_\ell^+(x))\right) .
\]
For every $j> M$ we obtain from~\eqref{martin} and~\eqref{annne}
\[
\frac{\nu_n(D^\pm_M(x))}{\lambda\lvert D^\pm_M(x)\rvert^{t_n}}
	\sum_{\ell=j}^n\lvert D^\pm_{\ell}(x)\rvert^{t_n}\le
	\sum_{\ell=j}^n\nu_n(D^\pm_\ell(x))
\le  \lambda\frac{\nu_n(D^\pm_M(x))}{\lvert D^\pm_M(x)\rvert^{t_n}}
	\sum_{\ell=j}^n\lvert D^\pm_{\ell}(x)\rvert^{t_n}
\]
and hence
\begin{multline} \label{e2ti}
\nu(\Delta_j(x)) \\\le
\lambda\left(
\frac{\nu(D^-_M(x))}{\lvert D^-_M(x)\rvert^{t_0}}
	\sum_{\ell=j}^\infty\lvert D^-_{\ell}(x)\rvert^{t_n}
	+
	\frac{\nu(D^+_M(x))}{\lvert D^+_M(x)\rvert^{t_0}}
	\sum_{\ell=j}^\infty\lvert D^+_{\ell}(x)\rvert^{t_n}\right), 
\end{multline}
together with the analogous lower bound.
If we know that each of the two series in \eqref{e2ti} converges then we can conclude that $\lim_{j\to\infty}\nu(\Delta_j(x))=0$ and hence that $\nu(x)=0$.

If $f$ is piecewise real analytic then we can expand it into a power series inside some one-sided neighbourhood of $x$:  

\[
f(y)=y + c (y-x)^b + O((y-x)^{b+1})
\]
for $y\in (x,x+\varepsilon)$ and similarly (possibly with different $c$ and $b$) in $(x-\varepsilon, x)$.
It implies
\begin{equation} \label{estimm}
C_1 \ell^{-\beta} \leq \lvert D^+_\ell(x)\rvert \leq C_2 \ell^{-\beta}
\end{equation}	
for $\beta=b/(b-1)$ and for some positive numbers $C_1, C_2$, for $\ell$ big enough.
We assume that $M$ was big enough so that \eqref{estimm} is true for all $\ell\geq M$.
 
Existence and finiteness of $t_0$-conformal measure $\nu$ together with \eqref{martin} implies that $\beta t_0>1$. Given $\varepsilon>0$ such that $\beta(t_0-\varepsilon)>1$, we obtain
\[
\sum_{\ell=M}^\infty\lvert D^+_\ell(x)\rvert ^{t_0-\varepsilon} < +\infty,
\]
hence the second series in \eqref{e2ti} converges.
The first series converges for the analogous reason and we are done.
\end{proof}

The real analyticity assumption is only necessary to obtain \eqref{estimm}.
We do not want to assume \eqref{estimm} outright, however, because we have \emph{a priori} no control on the location of the parabolic points.
	
We denote by $P^s$ the $s$-dimensional \emph{packing measure}.	
		
\begin{corollary} \label{cor:pt0}
	If $f$ is (piecewise) real analytic then $P^{t_0}(H)=P^{t_0}(\Lambda_Q)>0$.
\end{corollary}

\begin{proof}
The proof is analogous to Proposition \ref{prop:dimH}.
By Proposition~\ref{mmart}, there exists a $t_0$-conformal measure $\nu$ supported on $H$.
Let $x\in H$ and consider the sequences $(n_k)_k$ and $(r_k)_k$ and positive constants $c_1$ and $c_2$  associated to $x$.
We have

 \begin{equation}\label{mann}
 2 r_k\cdot \sup_{y\in B(x,r_k)}\lvert (f^{n_k})'(y)\rvert 
  \geq \diam f^{n_k}(B(x,r_k)) \ge c_1.
  \end{equation}
 By $t_0$-conformality of the measure $\nu$, we have
  \begin{eqnarray*}
  \nu(f^{n_k}(B(x,r_k))) 
  &=& \int_{B(x,r_k)}\lvert (f^{n_k})'(y)\rvert^{t_0} d\nu(y)\\
  &\ge& c_2^{-t_0}\sup_{y\in B(x,r_k)}\lvert (f^{n_k})'(y)\rvert^{t_0}
        \nu(B(x,r_k)).
  \end{eqnarray*}

With~\eqref{mann} we obtain
  \[
  1\ge \nu(f^{n_k}(B(x,r_k))) 
  \ge \left(c_1 c_2^{-1} r_k^{-1}\right)^{t_0}\nu(B(x,r_k))
  \]
  and hence
  \[
\liminf_{r\searrow 0} \frac {\nu(B(x,r))} {r^{t_0}} \leq \left(c_1^{-1} c_2 \right)^{t_0}.
  \]
  The assertion follows from the Frostman Lemma, see \cite{Fal:97}.
  \end{proof}

\begin{corollary}
	If $f$ is (piecewise) real analytic then $t_0<1$.
\end{corollary}

\begin{proof}
	This is an immediate consequence of Proposition~\ref{prop:zero} and Corollary~\ref{cor:pt0}. 
\end{proof}

We finally give the details to the example claimed in Section~\ref{sec:intro}. We demonstrate that the existence of a $t_0$-conformal measure which is supported on the complement of the parabolic points is not guaranteed if the map is not real analytic.
Consider the map 
\begin{equation}\label{ugu}
f(x)\eqdef 
\begin{cases} x+x^2e^{-1/x} &\text{ if }x\in[0,a]\\10x-9&\text{ if }x\in[\frac{9}{10},1]\end{cases},
\end{equation}
where $a\approx 0.8095$ is the solution of the equation  $a^2e^{-1/a}+a=1$.

Note that for $n\ge 1$ large enough we have 
\[
\frac{1}{\log(n-1)} 
\sim \frac{1}{\log n-\frac{1}{n} }
= \frac{1}{\log n\left(1-\frac{1}{n\log n}\right)}
\sim \frac{1}{\log n} + \frac{1}{n(\log n)^2} = f\left(\frac{1}{\log n}\right). 
\]
More precisely, 
\[
\frac{1}{\log (n-\frac 1 2)}<f\left(\frac{1}{\log n} \right)<\frac{1}{\log(n-2)}
\]
for $n$ big enough.
Hence, from~\eqref{martin} we derive that $\lvert D_k^+(0)\rvert \sim \frac{1}{k(\log k)^2}$ for every $k$ large enough. Obviously,
\[
\sum_{k=1}^\infty\lvert D^+_k(0)\rvert<+\infty,
\]
but, for any $t<1$ we have
\[
\sum_{k=1}^\infty\lvert D^+_k(0)\rvert^t =+\infty.
\]
and hence there cannot exist any $t$-conformal measure for $t<1$.
This implies that the Hausdorff dimension of the limit set $\Lambda_Q$ equals 1. However, the one-dimensional packing measure, which is equal to its  Lebesgue measure, is 0. We note that a slightly different version of the above example was studied in~\cite{Tha:83}.

\bibliographystyle{amsplain}

\end{document}